\documentclass[10pt]{amsart}
\usepackage[english]{babel}
\usepackage{amssymb}
\usepackage{amsmath}
\usepackage{lscape}

\newtheorem{dummy}{Dummy}

\newtheorem{theorem}[dummy]{Theorem}
\newtheorem{proposition}[dummy]{Proposition}
\newtheorem{corollary}[dummy]{Corollary}

\theoremstyle{definition}
\newtheorem{definition}{Definition}

\newtheorem{remark}[dummy]{Remark}

\newcommand{\ignore}[1]{}

\date{29.4.2019}
\author{S. Pumpl\"un}
\email{susanne.pumpluen@nottingham.ac.uk}
\address{School of Mathematical Sciences\\
University of Nottingham\\
University Park\\
Nottingham NG7 2RD\\
United Kingdom
}

\keywords{Skew polynomial, Ore polynomial, Azumaya algebra,  cyclic algebra.}

\subjclass[2010]{Primary: 17A35; Secondary: 17A60, 16S36}

\begin{document}

\title[Generalized cyclic Azumaya algebras]
{The automorphisms of generalized cyclic Azumaya algebras}

\begin{abstract}
We define a nonassociative generalization of cyclic Azumaya algebras employing skew polynomial rings $D[t;\sigma]$, where
 $D$ is an Azumaya algebra of constant rank with center $C$ and $\sigma$ an automorphism of $D$, such that $\sigma|_{C}$ has finite order.
The automorphisms of these algebras are canonically induced by ring automorphisms of the skew polynomial ring $D[t;\sigma]$ used in their construction. We  achieve a description of their inner automorphisms.
Results on the automorphisms of classical Azumaya algebras and central simple algebras of this type are obtained as special cases.
\end{abstract}

\maketitle

%
\section*{Introduction}
%

Let $D$ be an Azumaya algebra of constant rank with center $C$. Let $\sigma\in {\rm Aut}(D)$ be a ring automorphism, such that $\sigma|_{C}$ has finite order $m$ and fixed ring $S_0={\rm Fix}(\sigma)\cap C$. We assume that $C/S_0$ is a cyclic Galois ring extension of degree $m$ with Galois group $\mathrm{Gal}(C/S_0) = \langle \sigma |_{C} \rangle$.
We define generalized cyclic Azumaya algebras with the help of the skew polynomial rings $D[t;\sigma]$ as quotient algebras $D[t;\sigma]/D[t;\sigma]f$ for some $f(t)=t^m-d$, where $m$ is the order of $\sigma|_{C}$ and $d\in S_0=C\cap {\rm Fix}(\sigma)$.
Generalized cyclic Azumaya algebras are special examples of crossed product algebras that are Azumaya algebras: the crossed product is taken using the Azumaya algebra $D$ of constant rank and the cyclic group $\langle \sigma\rangle$. In particular, this construction yields cyclic Azumaya algebras when employing a Galois ring extension $S/S_0$ with cyclic Galois group $G=\langle \sigma\rangle$ for an automorphism $\sigma$ of $S$ with fixed ring $S_0$ when choosing the skew polynomial ring for the construction.

 This approach allows us to fit generalized Azumaya algebras into a more general family of nonassociative algebras using a construction that goes back to Petit \cite{P66}: if we take $f(t)=t^m-d$, where $m$ is the order of $\sigma|_{C}$, but allow $d\in D^\times$, we still obtain an algebra over $S_0$ but for $d\not\in S_0$, this algebra is not associative. More precisely, the algebra
$S_f=D[t;\sigma,\delta]/D[t;\sigma,\delta]f$ is defined on the additive subgroup
$\{h\in D[t;\sigma,\delta]\,|\, {\rm deg}(h)<m \}$ of $D[t;\sigma,\delta]$ by using right division by $f$
to define the algebra multiplication $g\circ h=gh \,\,{\rm mod}_r f $.
The algebras $S_f$ were introduced and studied in detail by Petit in
\cite{P66, P68} when $D$ is a division ring, and more generally for arbitrary rings in \cite{P15}.

Note that associative generalized cyclic algebras over fields were investigated by Amitsur in \cite{Am2} and nonassociative generalized cyclic algebras over fields in \cite{BP18}.
Another generalization of associative generalized cyclic algebras over fields called \emph{associative cyclic extensions of  simple rings} was considered by Kishimoto \cite{K}.

After introducing the basic terminology in Section \ref{sec:prel}, we introduce nonassociative generalized cyclic Azumaya algebras and some of their properties in Section \ref{sec:nonass}. In Section \ref{sec:gencyclic}, we then prove that every automorphism of a nonassociative generalized Azumaya algebra $A$ of constant rank over $S_0$ is canonically induced by some ring homomorphism of $D[t;\sigma]$ (Theorem \ref{thm:aut2}).  Furthermore, the automorphisms $H_{id_D,k}$ that extend  $id_D$ are in one-one correspondence with the elements of the group $\{c\in C\,|\, N_{C/S_0}(c)=1\}$.
 These are the only automorphisms of $A$, unless some $\tau\not=id_D$ that commutes with $\sigma$ can be extended to an $S_0$-automorphism of $A$ as well. In particular, every automorphism of an associative generalized cyclic Azumaya algebra $A$ of constant rank extends an automorphism $\tau\in {\rm Aut}_{S_0}(D)$ that commutes with $\sigma$, and the possible extensions $H_{\tau,k}$ of an $S_0$-automorphism $\tau$ that commutes with $\sigma$ are in one-one correspondence with the group $\{c\in C\,|\, N_{C/S_0}(c)=1\}$. Moreover, the $S_0$-automorphisms $\tau$ of $D$ that commute with $\sigma$ form a subgroup of the automorphism group of the Azumaya algebra.

 In Section \ref{sec:inner}, we show that all automorphisms of a nonassociative generalized cyclic Azumaya algebra can be written as a composition of an inner automorphism and a map that is a canonical extension of $\tau$ to $A$ of the type $\sum_{i=0}^{m-1}a_it^i\mapsto \sum_{i=0}^{m-1}\tau(a_i)t^i$ (Theorem \ref{prop:inner}).
 We also prove that the automorphisms $H_{id_D,k}$ of a nonassociative generalized cyclic Azumaya algebra which extend the identity $id_D$ are inner for all $k\in C$, such that there is $c\in C^\times$ with $k=c^{-1}\sigma(c)$.
In the special case that we have a nonassociative generalized cyclic algebra $(D,\sigma,d)$ over a base field, all its automorphisms are of this last type.

As an immediate consequence of our results, the automorphisms of an Azumaya algebra $(D,\sigma,d)$, where the ring extension satisfies an analogue of Hilbert's Theorem 90, are the composition of an inner automorphism $G_c$, $c\in C^\times$ with the canonical extension $H_{\tau,1}$ of some $\tau \in {\rm Aut}_{S_0}(D)$ which commutes with $\sigma$ (Corollary \ref{thm:generalizedcycliccsa}). In particular, the automorphisms of a generalized cyclic central simple algebra  over a field are induced by ring automorphisms of the ring $D[t;\sigma]$ used in their construction; each is the composition of an inner automorphism $G_c$, $c\in C^\times$ with the canonical extension $H_{\tau,1}$ of some $\tau \in {\rm Aut}_{S_0}(D)$ which commutes with $\sigma$ (Corollary \ref{cor:generalizedcycliccsa}).
If $A=(K/F,\sigma,d)$ is an associative cyclic central simple algebra, we show that all its automorphism can be described as inner automorphisms of the type $ G_{ct^j}$ for some $c\in K$ with $N_{K/F}(c)=1$ and a suitable integer $j$, $0\leq j \leq m-1$.

%
%

\section{Preliminaries} \label{sec:prel}


\subsection{Nonassociative algebras} \label{subsec:nonassalgs}


Let $R$ be a unital commutative ring and let $A$ be an $R$-module. $A$ is an
\emph{algebra} over $R$ if there exists an $R$-bilinear map $A\times
A\to A$, $(x,y) \mapsto x \cdot y$, denoted simply by juxtaposition
$xy$, the  \emph{multiplication} of $A$. An algebra $A$ is called
\emph{unital} if there is an element in $A$, denoted by 1, such that
$1x=x1=x$ for all $x\in A$. We will only consider unital algebras.

The {\it associator} of $A$ is given by $[x, y, z] =
(xy) z - x (yz)$. The {\it left nucleus} of $A$ is defined as ${\rm
Nuc}_l(A) = \{ x \in A \, \vert \, [x, A, A]  = 0 \}$, the {\it
middle nucleus} of $A$ is ${\rm Nuc}_m(A) = \{ x \in A \, \vert \,
[A, x, A]  = 0 \}$ and  the {\it right nucleus} of $A$ is
${\rm Nuc}_r(A) = \{ x \in A \, \vert \, [A,A, x]  = 0 \}$. ${\rm Nuc}_l(A)$, ${\rm Nuc}_m(A)$, and ${\rm Nuc}_r(A)$ are associative
subalgebras of $A$. Their intersection
 ${\rm Nuc}(A) = \{ x \in A \, \vert \, [x, A, A] = [A, x, A] = [A,A, x] = 0 \}$ is the {\it nucleus} of $A$.
${\rm Nuc}(A)$ is an associative subalgebra of $A$ containing $F1$
and $x(yz) = (xy) z$ whenever one of the elements $x, y, z$ lies in
${\rm Nuc}(A)$.   The
 {\it center} of $A$ is defined as ${\rm C}(A)=\{x\in A\,|\, x\in \text{Nuc}(A) \text{ and }xy=yx \text{ for all }y\in A\}$.

For a subring $B$ of a unital ring $A$, the \emph{centralizer} (also called the  \emph{commutator subring} if $A$ is associative) of $B$ in $A$ is defined as
${\rm Cent}_A(B)=\{a\in A\,|\, ab=ba \text{ for all } b\in B\}$.  
If ${\rm Cent}_A(B)=B$ then ${\rm Cent}_A(B)$ is a maximal commutative nonassociative subring of $A$.

An automorphism $G\in {\rm Aut}(A)$ of an algebra $A$ is called an \emph{inner automorphism}
if there is an element $m\in A$ with left inverse $m_l$ such that $G(x) = (m_lx)m$ for all $x\in A$. We denote such an automorphism by $G_m$. Given an inner automorphism
$G_m\in {\rm Aut}(A)$ and some $H \in {\rm Aut}(A)$, $H^{-1}\circ G_m \circ H\in {\rm Aut}(A)$ is also an inner automorphism.
Indeed, \cite[Lemma 2, Theorem 3, 4]{W09} generalize to any nonassociative algebra over a ring:

\begin{proposition} \label{prop:inner_Wene}
Let $A$ be an algebra over $R$.
\\ (i) For all invertible $n\in {\rm Nuc}(A)$, $G_n(x)=(n^{-1}x)n$ is an inner automorphism of $A$.
\\ (ii) If $G_m$ is an inner automorphism of $A$, then so is $G_{nm}(x) = ((m_l n^{-1})x)(nm)$ for all
invertible  $n\in {\rm Nuc}(A)$.
\\ (iii) If $G_m$ is an inner automorphism of $A$, and $a,b\in {\rm Nuc}(A)$ are invertible, then
$G_{am}=G_{bm}$ if and only if $ ab^{-1}\in C(A).$
\\ (iv) For invertible $n,m\in {\rm Nuc}(A)$, $G_m=G_n$ if and only if $n^{-1}m\in C(A)$.
\end{proposition}

The set $\{G_m\,|\, m\in {\rm Nuc}(A) \text{ is invertible} \}$ of inner automorphisms is a
subgroup of ${\rm Aut}_R(A)$. For each invertible $m\in {\rm Nuc}(A)\setminus C(A)$, $G_m$ generates a cyclic subgroup
of inner automorphisms which has finite order $s$ if $m^s\in C(A)$, so in particular if $m$ has order $s$.

If the nucleus of $A$ is commutative, then for all invertible $n\in
{\rm Nuc}(A)$, $G_n(x)=(n^{-1}x)n$ is an inner automorphism of $A$
such that ${G_n}|_{{\rm Nuc}(A)}=id_{{\rm Nuc}(A)}$.

\subsection{Azumaya algebras}

An algebra
$A$ is called an \emph{Azumaya} algebra over a unital ring $R$, if  $A$ is finitely generated as an $R$-module, $A$ is a separable extension of $R$, and the center of $A$ is $R$.
Equivalently, $A$ is an Azumaya algebra if $A$ is a finitely generated $R$-module and $A/Am$ a central simple $R/m$-algebra, for all maximal ideals $m$ in $R$.

A commutative ring extension $R'$ of $R$ is called a \emph{splitting ring} of $A$ if $A\otimes_R R'\cong {\rm End}_{R'}(P)$
 for a suitable faithfully projective $R'$-module $P$.

 If $S$ is a maximal commutative subalgebra in $A$ and $S$ is separable over $R$ then $S$ is a splitting ring for $A$ \cite{Kn}.

\subsection{Galois extensions}

Let $S$ be  a commutative ring. Recall that two $S_0$-algebra homomorphisms $\sigma$ and $\tau$ are called \emph{strongly distinct},
if for every non-zero idempotent $e\in S$ there is $x\in S$ such that $\sigma(x)e\not=\tau(x)e$.

 If we assume that $S$ is an $S_0$-algebra and faithfully projective as an $S_0$-module,
and that $G$ is a group of $S_0$-algebra homomorphisms of $S$, then the following conditions are equivalent:
\\ (i) $S$ is a separable $S_0$-algebra, finitely generated projective as an $S_0$-module, and ${\rm rank}_{S_0}S=|G|$.
All elements of $G$ are pairwise strongly distinct.
\\ (ii) $S$ is faithfully projective as an $S_0$-module, ${\rm rank}_{S_0}S=|G|$ and for every $\sigma\in G$
there exist $x_{i,\sigma},y_{i,\sigma}\in S$ such that $\sum_{i=1}^{m_\sigma}x_{i,\sigma}\tau(y_{i,\sigma})=
\delta_{\tau,\sigma}$ for all $\tau\in G$.
\\ (iii) $\varphi:S\otimes S\longrightarrow S^n,$ $\varphi(s\otimes t)=(s\tau(t))_{\tau\in G}$ is an isomorphism of $S$-algebras where $n$ is the order of $G$.

If one of these conditions is satisfied we call $S/S_0$ a \emph{Galois extension} with \emph{Galois group $G$}
in the sense of Chase-Harrison-Rosenberg \cite{CHR}.

Moreover, $S/S_0$ is called a \emph{weakly Galois extension} (resp., \emph{(G-)Galois} in Szeto's papers \cite{SzX, Sz, Sz12, Sz1}) if
 $S$ is a separable algebra over $S_0$, finitely generated projective as $S_0$-module, and
 there is a finite group of automorphisms $G$ of $S$ such that
  $S_0=S^G$ is the fixed ring of $G$ in $S$. Note that if $S/S_0$ is a Galois extension with Galois group $G$ then this implies that $S^G=S_0$.
  If $S/S_0$ is a Galois extension and the elements in $G$ are $S_0$-automorphisms of $S$ then $S$ is a weakly Galois extension of $S_0$ of constant rank (cf. \cite[Lemma 2.3, Corollary 2.4]{FP} for this summary).

  If $S/S_0$ is a Galois extension with Galois group $G$, then the map $H\mapsto S^H$ gives a one-one correspondence
  between the set of subgroups of $G$ and the set of $S_0$-subalgebras $S_H$ of $S$ which are separable and \emph{$G$-strong},
  i.e. the restriction of any two elements of $G$ to $S_H$ are either equal or strongly distinct as maps from $S_H$ to $S_H$.

  In this paper, we will use commutative Galois extensions that are weakly commutative Galois extensions.

\subsection{Skew polynomial rings}

Let $S$ be a unital associative ring, $\sigma$ a ring endomorphism of
$S$ and $\delta:S\rightarrow S$ a \emph{left $\sigma$-derivation},
i.e. an additive map such that
$\delta(ab)=\sigma(a)\delta(b)+\delta(a)b$
for all $a,b\in S$.
Then the \emph{skew polynomial ring} $S[t;\sigma,\delta]$ is the
set of skew polynomials $g(t)=a_0+a_1t+\dots +a_nt^n$ with $a_i\in
S$, with term-wise addition and where the multiplication is defined
via $ta=\sigma(a)t+\delta(a)$ for all $a\in S$. That means,
$$at^nbt^m=\sum_{j=0}^n a(\Delta_{n,j}\,b)t^{m+j}$$ for all $a,b\in
S$, where the map $\Delta_{n,j}$ is defined recursively via
$$\Delta_{n,j}=\delta(\Delta_{n-1,j})+\sigma (\Delta_{n-1,j-1}),$$ with
$\Delta_{0,0}=id_S$, $\Delta_{1,0}=\delta$, $\Delta_{1,1}=\sigma $.
Therefore $\Delta_{n,j}$ is the sum of all polynomials in $\sigma$
and $\delta$ of degree $j$ in $\sigma$ and degree $n-j$ in $\delta$
\cite[p.~2]{J96}. If $\delta=0$, then $\Delta_{n,n}=\sigma^n$.

For $\sigma=id$ and $\delta=0$, we obtain the usual ring of left
polynomials $S[t]=S[t;id,0]$. Define ${\rm Fix}(\sigma)=\{a\in S\,|\,
\sigma(a)=a\}$ and ${\rm Const}(\delta)=\{a\in S\,|\, \delta(a)=0\}$.

 For $f(t)=a_0+a_1t+\dots +a_nt^n\in S[t;\sigma,\delta]$ with $a_n\not=0$ define ${\rm deg}(f)=n$ and ${\rm deg}(0)=-\infty$.
Then ${\rm deg}(gh)\leq{\rm deg} (g)+{\rm deg}(h)$ (with equality if
$h$ has an
 invertible leading coefficient, or $g$ has an
 invertible leading coefficient and $\sigma$ is injective, or if $S$ is a division ring).
 An element $f\in S[t;\sigma,\delta]$ is \emph{irreducible} if it is not a unit and  it has no proper factors, i.e if there do not exist $g,h\in R$ with
 ${\rm deg}(g),{\rm deg} (h)<{\rm deg}(f)$ such
 that $f=gh$.


\subsection{Algebras obtained from skew polynomial rings}


Let $S$ be a unital ring and $R=S[t;\sigma,\delta]$ a skew polynomial
ring where $\sigma$ is injective.

Assume $f(t)=\sum_{i=0}^{m}a_it^i \in S[t;\sigma,\delta]$
 has an invertible leading coefficient $a_m\in S^\times$. Then
for all $g(t)\in S[t;\sigma,\delta]$ of degree $l\geq m$,  there exist  uniquely
determined $r(t),q(t)\in S[t;\sigma,\delta]$ with
 ${\rm deg}(r)<{\rm deg}(f)$, such that
$g(t)=q(t)f(t)+r(t)$ (\cite{CB},\cite[Proposition 1]{P15}).

Let ${\rm mod}_r f$ denote the remainder of right division by $f$.
The skew polynomials of degree less that $m$ canonically represent the elements of the right
$S[t;\sigma,\delta]$-module $S[t;\sigma,\delta]/S[t;\sigma,\delta]f$.
Moreover, $$\{g\in S[t;\sigma,\delta]\,|\, {\rm deg}(g)<m\}$$
 together with the multiplication
   \[g\circ h=
  \begin{cases}
  gh  \text{ if } {\rm deg} (g)+{\rm deg} (h) < m,\\
  gh \,\,{\rm mod}_r f \text{ if } {\rm deg} (g)+{\rm deg} (h) \geq m,
   \end{cases}
  \]
is a unital nonassociative ring $S_f$ also denoted by
$S[t;\sigma,\delta]/S[t;\sigma,\delta]f$. $S_f$ is a unital nonassociative algebra over the
commutative subring
$$\{a\in S\,|\, ah=ha \text{ for all } h\in S_f\}={\rm Comm}(S_f)\cap S$$
 of $S$.
 When the context is clear, we will drop the $\circ$ notation and simply use juxtaposition for multiplication in $S_f$.
Note that if $f$ has degree 1 then $S_f\cong S$, and if $f$ is reducible then $S_f$ contains zero divisors.
For all invertible $a\in S$ we have $S_f = S_{af}$.

 In the following, we assume $m\geq 2$ and call the algebras $S_f$  \emph{Petit algebras}
as the construction goes back to Petit \cite{P66}.

 $S_f$ is a free left $S$-module  of rank $m$ with basis $t^0=1,t,\dots,t^{m-1}$.
 $S_f$ is associative if and only if $Rf$ is a two-sided ideal in $R$.
 If $S_f$ is not associative then
$S\subset{\rm Nuc}_l(S_f),\,\,S\subset{\rm Nuc}_m(S_f)$ and $$\{g\in
R\,|\, {\rm deg}(g)<m \text{ and }fg\in Rf\}= {\rm Nuc}_r(S_f).$$
When $S$ is a division ring, these inclusions become equalities \cite[Theorem 4]{P15}.

 Note that $C(S_f)={\rm Comm}(S_f)\cap {\rm Nuc}_l(S_f)\cap {\rm Nuc}_m(S_f)\cap {\rm Nuc}_r(S_f)$ and
so
$$\{a\in S\,|\, ah=ha \text{ for all } h\in S_f\}={\rm Comm}(S_f)\cap S\subset C(S_f).$$
If ${\rm Nuc}_l(S_f)= {\rm Nuc}_m(S_f)=S$ this yields that the center
$C(S_f)={\rm Comm}(S_f)\cap S\cap {\rm Nuc}_r(S_f)={\rm Comm}(S_f)\cap S$ of $S_f$  is identical to the ring
$\{a\in S\,|\, ah=ha \text{ for all } h\in S_f\}.$
 Also note that
$$C(S)\cap{\rm Fix}(\sigma)\cap {\rm Const}(\delta)\subset \{a\in S\,|\, ah=ha \text{ for all } h\in S_f\}$$
which is proved analogously as \cite[Theorem 8 (ii)]{P15}, where $\delta=0$.

If $\delta=0$ and $f(t)=\sum_{i=0}^{m}a_it^i \in S[t;\sigma]$ then
 ${\rm Comm}(S_f)=\{g\in S_f\,|\, gh=hg \text{ for all } h\in S_f\}$ contains the set
 $$\{\sum_{i=0}^{m-1}d_it^i \,|\, d_i\in {\rm Fix}(\sigma)\text{ and } cd_i=d_i\sigma^i(c)\text{ for all } c\in S \}$$
 \cite[Theorem 8 (i)]{P15}.
If $a_0$ is invertible, the two sets are equal \cite[Proposition 7.4 (iii)]{CB}, implying
$$\{a\in S\,|\, ah=ha \text{ for all } h\in S_f\}=\{a\in S \,|\, a\in {\rm Fix}(\sigma)\text{ and } ca=ac\text{ for all } c\in S \}
=C(S)\cap{\rm Fix}(\sigma).$$

%
%

\section{Nonassociative generalized cyclic Azumaya algebras} \label{sec:nonass}

In the following, let $D$ be a unital ring with center $C=C(D)$.
Let $\sigma\in {\rm Aut}(D)$ be a ring automorphism, such that $\sigma|_{C}$ has finite order $m$ and fixed ring $S_0={\rm Fix}(\sigma)\cap C$. Moreover, we assume that $C/S_0$ is a cyclic Galois ring extension of degree $m$ with Galois group $\mathrm{Gal}(C/S_0) = \langle \sigma |_{C} \rangle$. This implies that $C$ has constant rank $m$ as an $S_0$-module.

Let  $f(t)=t^m-d\in D[t;\sigma]$ with $d\in D^\times$ invertible.  Then $S_f=D[t;\sigma]/D[t;\sigma](t^m-d)$ is a nonassociative algebra over its center
$$\{a\in D\,|\, ah=ha \text{ for all } h\in S_f\}=C\cap{\rm Fix}(\sigma)=S_0.$$
Moreover, since $d$ is invertible we know that
$${\rm Comm}(S_f)=\{\sum_{i=0}^{m-1}d_it^i \,|\, d_i\in {\rm Fix}(\sigma)\text{ and } cd_i=d_i\sigma^i(c)\text{ for all } c\in S\}.$$

 In the associative setting, i.e. when $d\in S_0^\times$, we  have the following result:

\begin{theorem} \label{thm:mainAzumaya}
Let $S_f=D[t;\sigma]/D[t;\sigma](t^m-d)$ with $d\in S_0^\times$. Then the following statements are equivalent:
 \\ (i) $D[t;\sigma]/D[t;\sigma](t^m-d)$ is an Azumaya algebra over $S_0$.
 \\ (ii) $D$ is an Azumaya algebra over $C$.
\end{theorem}

\begin{proof}
(ii) implies (i):  Since $D$ is an Azumaya algebra over $C$, $D$ is separable over $C$. Moreover, $S_f$ is a separable extension of $D$, since
$$x=m^{-1}d^{-1}\sum_{i=1}^{m-1}t^i\otimes t^{m-1-i}$$
 is a separable idempotent of $S_f$ over $D$: we have $m^{-1}d^{-1}\sum_{i=1}^{m-1}t^i\otimes_D t^{m-1-i}=1$ and
 $gx=xg$ for all $g\in S_f$. Thus $S_f$ is a separable extension of $S_0$
by the transitivity of separable extensions. By our general theory on Petit algebras, we know that
$S_f$ is an algebra with center $S_0$.  By construction, $S_f$ is finitely generated as an $S_0$-module, since $D$ is finitely generated over $C$. Therefore $S_f$ is an Azumaya algebra over $S_0$.
\\ (i) implies (ii): Assume that $D[t;\sigma]/D[t;\sigma](t^m-d)$ is an Azumaya algebra. Then its center is $S_0=C\cap {\rm Fix}(\sigma)$. By Proposition \ref{prop:cent} below, the centralizer  ${\rm Cent}_{(D,\sigma, d)}(C)$ of $C$  in $(D,\sigma, d)$ is $D$. Thus $D$ is a separable $S_0$-algebra by the Commutator Theorem for Azumaya algebras \cite[Theorem 4.3]{ISz}, since $C$ is a separable $S_0$-algebra. Therefore $D$ is an Azumaya algebra.
\end{proof}

Assuming additionally that $\sigma$ has finite order $m$, we obtain:

\begin{theorem} \label{thm:mainAzumayaII}
Let $S_f=D[t;\sigma]/D[t;\sigma](t^m-d)$ with $d\in S_0^\times$ and let $\sigma$ have finite order $m$. Then the following statements are equivalent:
 \\ (i) $D[t;\sigma]/D[t;\sigma](t^m-d)$ is an Azumaya algebra over $S_0$.
 \\ (ii) Let $G_t(g)=tgt^{-1}$ be the inner automorphism of $S_f$ defined by $t$. Then $\{ u\in S_f \,|\, G_t(u)=u \}={\rm Fix}(\sigma)[t]/{\rm Fix}(\sigma)[t](t^m-d)$ is an Azumaya algebra
 with center $S_0[t]/S_0[t](t^m-d)$.
\end{theorem}

\begin{proof} We only need to show that $\{ u\in S_f \,|\, G_t(u)=u \}={\rm Fix}(\sigma)[t]/{\rm Fix}(\sigma)[t](t^m-d)$. The rest is proved in \cite[Theorem 3.3]{SzX}.
For this we note that both ${\rm Fix}(\sigma)\subset \{ u\in S_f \,|\, G_t(u)=u \}$ and $t,t^2,\dots,t^{m-1} \in \{ u\in S_f \,|\, G_t(u)=u \}$. Moreover, we have $t^m=d\in \{ u\in S_f \,|\, G_t(u)=u \}$, so that  $\{ u\in S_f \,|\, u(t)\in {\rm Fix}(\sigma)[t] \}\subset {\rm Fix}(\sigma)[t]$. This proves the assertion.
\end{proof}

More generally, we still have:

\begin{theorem} \label{thm:mainAzumayaIII}
Let $D$ be an Azumaya algebra over $C$ and let $\sigma|_C$ have finite order $m$.
Consider the Azumaya algebra $S_f=D[t;\sigma]/D[t;\sigma](t^m-d)$, $d\in S_0^\times$.  Let  $G_t(g)=tgt^{-1}$ be the inner automorphism of $S_f$ defined by $t$. Then:
\\ (i) $D[t;\sigma]/D[t;\sigma](t^m-d)$ is a $G_t$-Galois extension of $S_f^{G_t}=\{ u\in S_f \,|\, G_t(u)=u \}$, such that $S_f^{G_t}$ is a direct summand of $D[t;\sigma]/D[t;\sigma](t^m-d)$ as a bimodule over $S_f^{G_t}$.
\\ (ii)  $\{ u\in S_f \,|\, G_t(u)=u \}={\rm Fix}(\sigma)[t]/{\rm Fix}(\sigma)[t](t^m-d)$ is an Azumaya algebra
 with center $S_0[t]/S_0[t](t^m-d)$.
\end{theorem}

\begin{proof}
 (i) The proof of \cite[Theorem 3.3]{SzX} for this statement carries over verbatim. Note that
 $t$ is invertible in $S_f$ with inverse $t^{m-1}d^{-1}$.
 \\ (ii) The corresponding part of the proof of \cite[Theorem 3.1]{SzX} does not need the assumption that
$\sigma\in {\rm Aut}(D)$ has finite order $m$, since we assume that $d\in S_0^\times$, so in particular that $d\in C$, and we assume that $\sigma|_C$ has finite order order $m$.
\end{proof}

\begin{definition}
Let $D$ be an Azumaya algebra of constant rank with center $C$ and $f(t)=t^m-d\in D[t;\sigma]$ with $d\in S_0^\times$.
Then the associative algebra
$S_f=D[t;\sigma]/D[t;\sigma]f$ in Theorem \ref{thm:mainAzumaya} is called a \emph{generalized cyclic Azumaya algebra}.
We write $(D,\sigma, d)$ for this algebra.
\end{definition}

We follow Jacobson's terminology \cite[p.~19]{J96}. Note that for  $d\in S_0^\times$, $(D,\sigma, d)$ can also be viewed as a crossed product algebra.

In particular, if $D=C$ is a commutative ring, $C/S_0$ is a cyclic Galois extension of degree $m$ with Galois group generated by $\sigma$, and $ f(t)=t^m-d\in S_0[t]$, then we obtain an associative Azumaya algebra we will denote by $(C/S_0,\sigma,d)=C[t,\sigma]/C[t;\sigma]f$, which we call a \emph{cyclic Azumaya algebra}. Note that if $C/S_0$ is a cyclic field extension, then
$(C/S_0,\sigma,d)$ is an associative cyclic algebra over $S_0$ of degree $m^2$.

\begin{remark}
(i)
This construction of Azumaya algebras was first mentioned in \cite{PS} for  $f(t)=t^2-1\in D[t;\sigma]$ with $D$ a commutative ring, as a generalization of the classical quaternion algebra.
It was shown that $S_f=(D,\sigma, d)$ is an Azumaya algebra over $S_0$, that $D$ is a maximal commutative subalgebra of $(D,\sigma, d)$ \cite[Lemma 3.1]{Sz1}, and that $(D,\sigma, d)$ is a separable extension over ${\rm Fix}(\sigma)$. The centralizer of $D$ in $(D,\sigma, d)$ given by $\{g\in S_f\,|\, gs=sg \text{ for all } s\in D\}$ was shown to be $D$. Moreover, $S\otimes_{S_0} (D,\sigma, d) \cong {\rm Mat}_n(S)$.
(In \cite{Sz}, $(D,\sigma, d)$  is called a \emph{generalized quaternion ring extension.})
\\ (ii) The associative setting in \cite{SzX} is more restrictive than the one we consider: in \cite{SzX}, the ring automorphism $\sigma\in {\rm Aut}(D)$ is always required to have finite order $m$, and $d\in {\rm Fix}(\sigma)$.
 Since $m$ is the order of $\sigma$, in that case $t^mb=\sigma^m(b)t^m=bt^m$ for all $b\in D$, which implies $db=bd$ in $S_f$ for all $b\in D$, hence $d\in C$, i.e. $d\in C\cap {\rm Fix}(\sigma)$.
\\ (iii) If $D$ is a central simple algebra over a field $F$ of degree $n$, then $(D,\sigma, d)$ is a central simple algebra over the field $ S_0$ of degree $mn$ and the centralizer of $D$ in $(D,\sigma, d)$ is $F$ \cite[p.~20, Proposition 1.4.4]{J96}.
 \\ (iv) If $m$  is an invertible integer in $D$, then $S_0[t]/S_0[t](t^m-d)$ is a separable extension of $S_0$  contained in $A=(D,\sigma,d)$ \cite{SzX}.
\end{remark}

Our definition of associative generalized cyclic Azumaya algebras generalizes to nonassociative algebras as follows:

\begin{definition}
Let $D$ be an Azumaya algebra over $C$ and $f(t)=t^m-d\in D[t;\sigma]$, $d\in D^\times$. The algebra $(D,\sigma, d)=D[t;\sigma]/D[t;\sigma]f$ over $S_0$ is called a \emph{nonassociative generalized  cyclic Azumaya algebra}.
\end{definition}

In particular, if $D$ has constant rank $n$, then the algebra
$A=(D, \sigma, d)$, $d\in D^\times$, is finitely generated as an $S_0$-module of constant rank $m^2n^2$.

If  $S$ is a commutative ring, and $S/S_0$ is a cyclic Galois extension of degree $m$ with Galois group generated by
$\sigma$ and $ f(t)=t^m-d\in S[t;\sigma]$, we call the algebra $S[t,\sigma]/S[t;\sigma]f$ a \emph{nonassociative cyclic Azumaya algebra} and denote it by $(S/S_0,\sigma,d)$ (this is the case that $D=C$).

The following generalizes \cite[Lemma 3.2]{SzX} to the nonassociative setting:

\begin{proposition}\label{prop:cent}
For all $d\in D^\times$, the centralizer  ${\rm Cent}_{(D,\sigma, d)}(C)$ of $C$  in $(D,\sigma, d)$ is $D$.
\end{proposition}

\begin{proof} The proof is analogous to the one in the associative case when $d\in S_0^\times$:
It is easy to see that $D\subset {\rm Cent}_{(D,\sigma, d)}(C)=\{a(t)\in (D,\sigma, d)\,|\, a(t)c=ca(t) \text{ for all } c\in C\}$.
Conversely, for each $\sum_{i=0}^{m-1}a_it^i\in {\rm Cent}_{(D,\sigma, d)}(C)$, we know that $c(\sum_{i=0}^{m-1}a_it^i)=(\sum_{i=0}^{m-1}a_it^i)c$ for all $c\in C$, implying that $a_i(c-\sigma^i(c))=0$ for all $c\in C$ and for all $i.$ Since by assumption $C$ is a cyclic Galois extension of $S_0$, the ideal of $C$ generated by $\{c-\sigma^i(c)\,|\,c\in C\}$ is $C$. This means $a_i=0$ for all $i>0$ and so $\sum_{i=0}^{m-1}a_it^i=a_0\in D$, yielding ${\rm Cent}_{S_f}(C)\subset D$.
\end{proof}

The general structure theory of Petit algebras gives us the following two results:

\begin{theorem}
Let  $(S/S_0,\sigma,d)$ be a nonassociative cyclic Azumaya algebra over $S_0$.
\\ (i) $(S/S_0,\sigma,d)$ is finitely generated as an $S_0$-module of constant rank $m^2$.
\\ (ii) $(S/S_0,\sigma,d)$  is associative if and only if  $d\in  S_0$. If $(S/S_0,\sigma,d)$ is not associative then
$(S/S_0,\sigma,d)$ contains $S$ in its left and middle nucleus, and if $S_0$ is a domain, then ${\rm Nuc}_l((D,\sigma,d))={\rm Nuc}_m((D,\sigma,d))=S$.
\\ (iii) Let $s$ be the smallest integer such that $d\in {\rm Fix}(\sigma^s)$. Then $rs=m$ for some integer $r$ and
 the finitely generated $S_0$-module $S\oplus St^s\oplus\dots\oplus St^{(r-1)s}$ of rank $mr$ lies in ${\rm Nuc}_r((S/S_0,\sigma,d))$. If $s\not=1$ is a prime or $S_0$ is a domain, then
 ${\rm Nuc}_r((S/S_0,\sigma,d))=S\oplus St^s\oplus\dots\oplus St^{(r-1)s}$.
\\ (iv) $S$ is a maximal commutative nonassociative subring of $(S/S_0,\sigma,d)$.
\end{theorem}

\begin{proof}
 (i) is trivial as $S/S_0$ has constant rank $m.$
\\ (ii) By our general theory on Petit algebras, we know that
$(S/S_0,\sigma,d)$ is an algebra with center $S_0$ satisfying these properties, cf. \cite[Theorem 4]{P15}.
 \\ (iii) We know that ${\rm Nuc}_r((S/S_0,\sigma,d))=\{g\in R\,|\, {\rm deg}(g)<m \text{ and }fg\in Rf\}$ by \cite[Theorem 4]{P15}. We have ${\rm Nuc}_r((S/S_0,\sigma,d))\cap S=S$. Since $d\in {\rm Fix}(\sigma^s)$ we can easily conclude that also $t^s\in {\rm Nuc}_r((S/S_0,\sigma,d))$. This implies that the finitely generated $S_0$-module $S\oplus St^s\oplus\dots\oplus St^{(r-1)s}$ lies in the right nucleus of  $(S/S_0,\sigma,d)$. Since $S/S_0$ has constant rank $m$ by our assumptions, this submodule has constant rank $mr$. Now ${\rm Nuc}_r((S/S_0,\sigma,d))$ is an $S_0$-submodule of $(S/S_0,\sigma,d)$ and $(S/S_0,\sigma,d)$ has rank $m^2$. Hence comparing ranks we conclude that if $s$ is prime then
 either $r=m$ and $S_f$ is an associative algebra, or $r<m$, and ${\rm Nuc}_r((S/S_0,\sigma,d))=S\oplus St^s\oplus\dots\oplus St^{(r-1)s}$.

 If $S_0$ is a domain, we also obtain that ${\rm Nuc}_r((S/S_0,\sigma,d))=S\oplus St^s\oplus\dots\oplus St^{(r-1)s}$ by the same proof as the one of \cite[Proposition 3.2.3]{St} (note that there this is proved for the opposite algebra).
\\ (iv)
This generalizes \cite[Lemma 3.1]{Sz1} and follows immediately from Proposition \ref{prop:cent} above:
For all $d\in S^\times$, the centralizer of $S$ in $(S/S_0,\sigma,d)$ is  ${\rm Cent}_{(D,\sigma, d)}(S)=S$.
\end{proof}

Hence $S$ is always contained in the nucleus of $(S/S_0,\sigma,d)$ and if $S_0$ is a domain, then the nucleus of $(S/S_0,\sigma,d)$ is $S$.

\begin{theorem}  \label{thm:maingenAzumaya}
Let $(D,\sigma,d)$ be a nonassociative generalized cyclic Azumaya algebra of constant rank $n^2m^2$, $S_0={\rm Fix}(\sigma)\cap C$.
\\ (i)  $(D,\sigma,d)$ is finitely generated as an $S_0$-module.
\\ (ii) $(D,\sigma,d)$  is associative if and only if  $d\in D\setminus S_0$. If $(D,\sigma,d)$ is not associative then
$(D,\sigma,d)$ contains $D$ in its left and middle nucleus, and if $S_0$ is a domain, then ${\rm Nuc}_l((D,\sigma,d))={\rm Nuc}_m((D,\sigma,d))=D$.
\\ (iii)  Let $s$ be the smallest integer such that $d\in {\rm Fix}(\sigma^s)$, then either $m=rs$ for some integer $r$ and
the left $S_0$-module
$$C\oplus C t^s\oplus\dots\oplus C t^{(r-1)s}$$
of constant rank $rm$ lies in ${\rm Nuc}_r((D,\sigma,d))$
or $m=rs+b$ for two integers $r$ and $b$ with $0< b< s$ and the left $S_0$-module
$$C\oplus C t^s\oplus\dots\oplus C t^{rs}$$
of constant rank $(r+1)m$ lies in $ {\rm Nuc}_r((D,\sigma,d)).$ In particular, $S_0\subset {\rm Nuc}((D,\sigma,d)).$
\\ (iv)  $D$ is separable over $S_0$.
\end{theorem}

\begin{proof}
(i) This is proved analogously to \cite[Lemma 3.2]{SzX}.
\\ (ii) By our general theory on Petit algebras, we know that $(D,\sigma,d)$ is an algebra with center $S_0$ and the nuclei as claimed \cite[Theorem 4]{P15}.
\\ (iii) We have ${\rm Nuc}_r((D,\sigma,d))=\{g\in R\,|\, {\rm deg}(g)<m \text{ and }fg\in Rf\}$.
It is easy to check that $C\subset {\rm Nuc}_r((D,\sigma,d))$.
Let $s$ be the smallest integer such that $d\in {\rm Fix}(\sigma^s)$. Then a straightforward calculation shows that
$t^s\in {\rm Nuc}_r((D,\sigma,d))$. Hence $C\oplus C t^s\oplus\dots\oplus C t^{(r-1)s}\subset {\rm Nuc}_r((D,\sigma,d))$ if $m=rs$
and $C\oplus C t^s\oplus\dots\oplus C t^{rs}\subset {\rm Nuc}_r((D,\sigma,d))$ if $m=rs+b$.
\\ (iv) By construction, $(D,\sigma,d)$ is free of rank $m$ as a left $D$-module. Since $D$ is finitely generated as a $C$-module by assumption, and since $C$ is finitely
 generated as an $S_0$-module as it is a Galois extension, $(D,\sigma,d)$ is finitely generated as an $S_0$-module.
\\ (iv) is trivial by the assumptions on $D$ and $C/{\rm Fix}(\sigma)$.
\end{proof}

%
%

\section{Automorphisms of nonassociative generalized cyclic Azumaya algebras} \label{sec:gencyclic}

\subsection{Nonassociative generalized cyclic Azumaya algebras}
Let $D$ be an Azumaya algebra over $C$ of constant rank $n$ and $A = (D, \sigma, d)$ be a nonassociative generalized cyclic Azumaya algebra over $S_0=C\cap {\rm Fix}(\sigma)$. In the following $\sigma$ sometimes also stands for $\sigma|_C$ to simplify notation.

For $k\in C$, we define  $N_{C/S_0}:C\longrightarrow S_0$ via
 $$N_{C/S_0}(k)=\Big( \prod_{l=0}^{m-1}\sigma^l(k) \Big).$$
Some of the ring automorphisms of the skew polynomial ring $D[t;\sigma]$ canonically  induce algebra automorphisms of $(D,\sigma,d)$:

 \begin{theorem} \label{thm:automorphism_of_Sf_field_case}
 Let  $A=(D,\sigma,d)$ be a nonassociative generalized cyclic Azumaya algebra.
Let $\tau\in {\rm Aut}_{S_0}(D)$ be an algebra automorphism that commutes with $\sigma$. For all $k \in C^{\times}$ such that
\begin{equation} \label{eqn:neccessary}
\tau(d) = \Big( \prod_{l=0}^{m-1}\sigma^l(k) \Big) d,
\end{equation}
define  $H_{\tau , k}:(D,\sigma,d)\longrightarrow (D,\sigma,d)$ via
$$
H_{\tau , k}(\sum_{i=0}^{m-1} x_i t^i )=\sum_{i=0}^{m-1} \tau(x_i) (kt)^i $$
$$= \tau(x_0) + \tau(x_1)kt +
\tau(x_2)k\sigma(k)t^2+\cdots +
\tau(x_{m-1}) k \sigma(k) \cdots \sigma^{m-2}(k) t^{m-1}.
$$
 Then $H_{\tau , k}$ is an automorphism of $A$ that extends $\tau$. $H_{\tau , k}$ is canonically induced by a ring automorphism of $D[t;\sigma]$.
 \\ (ii) For all $k \in C^{\times}$  such that $N_{C/S_0}(k) =1$,
$$H_{id,k}(\sum_{i=0}^{m-1}a_it^i)=  a_0 + \sum_{i=1}^{m-1} a_i \big( \prod_{l=0}^{i-1}\sigma^l(k) \big) t^i$$
is an automorphism of $(D,\sigma,d)$ extending $id_{D}$.
\end{theorem}

\begin{proof}
(i) Let $G$ be a ring automorphism of $D[t;\sigma]$. Then for $h(t) = \sum_{i=0}^{m-1} b_i t^i \in D[t; \sigma]$ we have
$$G(h(t))=\tau(b_0)  +\sum_{i=i}^{m-1}\tau(b_i) \prod_{l=0}^{i-1}\sigma^l(k) t^i$$
for some $\tau\in{\rm Aut}(D)$  such that $\sigma\circ\tau=\tau\circ\sigma$ and some $k\in C^\times$
(the proof of  \cite[p.~75]{K} works for $D[t;\sigma]$). It is straightforward to see that for $\tau\in {\rm Aut}_{S_0}(D)$  and  $S_f=(D, \sigma, d)$, we have $S_f\cong S_{G(f)}$ (cf. \cite[Theorem 7]{LS} or \cite[p. 55 ff.]{CB}, the proofs also work when $D$ is not a division algebra).
In particular, this means that if $k\in C^\times$ satisfies  \eqref{eqn:neccessary} then
$G(f(t)) = \big( \prod_{l=0}^{m-1} \sigma^l(k) \big) f(t)=af(t)$ with $a\in D^\times$ being the product of the $\sigma^l(k)$, and so $G$ induces an isomorphism of $S_f$ with $ S_{af}=S_f$, i.e. an automorphism of $S_f=(D, \sigma, d)$.
\\ (ii) follows from (i).
\end{proof}

All the automorphisms of a nonassociative generalized cyclic Azumaya algebra $(D, \sigma, d)$ are canonically induced by ring automorphisms of the twisted polynomial ring $D[t;\sigma]$, i.e.  the maps $H_{\tau,k}$ are all possible automorphisms:

\begin{theorem} \label{thm:aut2}
Let $A = (D, \sigma, d)$, $d \in  D^\times$, be a nonassociative generalized cyclic Azumaya algebra of constant rank $m^2n^2$
over $S_0$.
\\ (i) Let $H \in {\rm Aut}_{S_0}(A)$. Then $H=H_{\tau,k}$ with
\begin{equation} \label{eqn:neccessaryII}
H_{\tau,k}( \sum_{i=0}^{m-1} a_i t^i )=\tau(a_0)+\sum_{i=1}^{m-1}\tau(a_i) \big( \prod_{l=0}^{i-1} \sigma^l(k) \big) t^i,
\end{equation}
for some $\tau\in {\rm Aut}_{S_0}(D)$ which commutes with $\sigma$, and some $k \in C^\times$  such that
 $\tau(d) = N_{C/S_0}(k) d$. All maps $H_{\tau,k}$ where $\tau \in {\rm Aut}_{S_0}(D)$ commutes with $\sigma$ and where $k \in C^\times$  such that
 $\tau(d) = N_{C/S_0}(k) d$ 
 are automorphisms of $A$, and induced by automorphisms of $D[t;\sigma]$.
\\ (ii) For all $k\in C^\times$ such that $N_{C/S_0}(k) = 1 $, $id\in {\rm Aut}(D)$ extends to an automorphism $H=H_{id,k}\in {\rm Aut}_{S_0}(A)$,
$$H_{id,k}( \sum_{i=0}^{m-1}a_i t^i)= a_0 +\sum_{i=1}^{m-1} a_i \big( \prod_{l=0}^{i-1} \sigma^l(k) \big) t^i,$$
\end{theorem}

The proof is similar to the one of \cite[Theorem 6]{BP18}, but works also when the algebra is associative, as it does not rely on the right nucleus being $D$, whereas the proof of \cite[Theorem 6]{BP18} did.

\begin{proof}
(i) Let $H \in {\rm Aut}_{S_0}(A)$. Then $H|_D\in {\rm Aut}_{S_0}(D)$, since $H$ leaves the commutator ${\rm Cent}_{(D,\sigma, d)}(C)$ invariant and we know that ${\rm Cent}_{(D,\sigma, d)}(C)=D$. (The argument in the proof of \cite[Theorem 6]{BP18} uses instead that the right nucleus is invariant under $H$).

Thus $H|_D = \tau$ for some $\tau\in {\rm Aut}_{S_0}(D)$. Write
$H(t) = \sum_{i=0}^{m-1} k_i t^i$ for some $k_i \in D$, then we have
$$H(tz) = H(t)H(z) = \big( \sum_{i=0}^{m-1} k_i t^i \big) \tau(z) = \sum_{i=0}^{m-1} k_i \sigma^i(\tau(z)) t^i,$$
and
$$H(tz) = H(\sigma(z)t) = \tau(\sigma(z)) \sum_{i=0}^{m-1} k_i t^i = \sum_{i=0}^{m-1} \tau(\sigma(z)) k_i t^i$$
for all $z\in D$. Comparing the coefficients of $t^i$ yields
$$k_i \sigma^i(\tau(z))  = \tau(\sigma(z)) k_i \text{ for all } i =\{0,\dots,m-1\}$$
for all $z \in D$. This implies that
$$k_i (\sigma^i(\tau(z)) - \tau(\sigma(z))) = 0 \text{ for all } i  \in \{0,\dots,m-1\}$$
for all $z \in C$.
 Now $\sigma$ restricted to $C$ generates the Galois group of the cyclic Galois extension $C/S_0$ by assumption, and $\tau|_C:C\longrightarrow C$ fixes $S_0$ by assumption, thus lies in this Galois group. Hence $\tau|_C$ commutes with $\sigma|_C$ and we obtain
$$k_i (\sigma^i(\tau(z)) - \sigma(\tau(z))) = 0 \text{ for all } i  \in \{0,\dots,m-1\}$$
for all $z \in C$, therefore
$$k_i (\sigma^{i-1}(w) - w) = 0 \text{ for all } i  \in \{1,\dots,m-1\}$$
for all $w\in C$.
 As $\sigma \vert_C$ has order $m$, we know that
 $\sigma \vert_C \not= \sigma^i \vert_C$ for all $1\not= i  \in \{0, \dots, m-1\}$.

 Since $C/S_0$ is Galois, we also know that the ideal of $C$ generated by $\{c-\sigma^i(c)\,|\, c\in C\}$ is all of $C$
  by \cite[p. 80]{DI}. That means $k_i = 0$ for all $1\not=i \in \{ 0, \ldots, m-1 \}$.
 For $i = 1$, we obtain $k_1 \tau(\sigma(z)) = \tau(\sigma(z)) k_1$
for all $z\in D$, hence $k_1 \in C$. This implies $H(t) = kt$ for some $k \in C^{\times}$.

Since
$$H(z t^i) = H(z) H(t)^i = \tau(z) (kt)^i = \tau(z) \Big( \prod_{l=0}^{i-1} \sigma^l(k) \Big) t^i,$$
for all $i \in \{ 1, \ldots, m-1 \}$ and all $z \in D$, $H$ has the form
$$H_{\tau,k}: \sum_{i=0}^{m-1} a_i t^i \mapsto \tau(a_0) + \sum_{i=1}^{m-1} \tau(a_i) \big( \prod_{l=0}^{i-1} \sigma^l(k) \big) t^i,$$
for some $k \in C^{\times}$.

Comparing the constant terms in $H(t)^m = H(t^m) = H(d)$ implies
$$\tau(d) = k \sigma(k) \cdots \sigma^{m-1}(k) d = N_{C/S_0}(k) d.$$

The fact that $H_{\tau,k}$ is by assumption a multiplicative map forces $\sigma$ and $\tau$ to commute:
$H_{\tau , k}(t\circ  c)= H_{\tau , k}(t)  H_{\tau , k}(c)$ for all $c\in D$ implies that
$$H_{\tau , k}(t\circ  c)= H_{\tau , k}(\sigma(c)t)=\tau(\sigma(c))k t$$
 and
$$H_{\tau , k}(t)  H_{\tau , k}(c)=kt \circ \tau(c)=k \sigma(\tau(c))t.$$
Thus we obtain $\tau(\sigma(c))k=k \sigma(\tau(c))$ for all $c\in D$ and some $k\in C^\times$, implying
$\tau(\sigma(c))=\sigma(\tau(c))$ for all $c\in D$.

By Theorem \ref{thm:automorphism_of_Sf_field_case} all these maps are automorphisms induced by automorphisms of the skew polynomial ring, since $\sigma$ and $\tau$ commute.
\\ (ii) For $\tau=id_D$, $H$  has the form
$$H_{id,k} (\sum_{i=0}^{m-1}a_i t^i )= a_0+\sum_{i=1}^{m-1} a_i \big( \prod_{l=0}^{i-1} \sigma^l(k) \big) t^i$$
for some $k\in C^\times$ with $k\sigma(k)\cdots \sigma^{m-1}(k) = N_{C/S_0}(k) = 1$ by (i).
\end{proof}

 It is clear that $H_{\tau,k}=H_{\rho,l}$ if and only if $\sigma=\rho$ and $k=l$, and that $H_{\tau,k}\circ H_{\rho,l}=H_{\tau\circ \rho,kl}$.

 \begin{corollary}\label{cor:7}
 (i) The subgroup of $S_0$-automorphisms of  $(D,\sigma,d)$
extending $id_D$ is isomorphic to
$$\{ k \in C^\times\,|\, N_{C/S_0}(k) = 1 \}.$$
 \\ (ii) Suppose that $S_0$ contains an $m$th root of unity $\omega$. If  $\tau$ has finite order $s$, then the cyclic subgroup $\langle H_{\tau,\omega}\rangle$ of ${\rm Aut}_{S_0}(A)$  generated by $H_{\tau,\omega}$ has order at most $ms$.
(If $\tau$ has infinite order then this subgroup has infinite order.)
In particular, $\langle H_{id,\omega}\rangle$
is a cyclic subgroup of ${\rm Aut}_{S_0}((D,\sigma,d))$ of order at most $m$.
 \end{corollary}

As a consequence of Theorem \ref{thm:aut2} and Corollary \ref{cor:7} we obtain all automorphisms of a nonassociative cyclic Azumaya algebra:

\begin{theorem} \label{thm:aut3}
Let $A=(S/S_0,\sigma,d)$ be a nonassociative cyclic Azumaya algebra.
\\ (i) Let $H \in {\rm Aut}_{S_0}(A)$. Then $H=H_{\tau,k}$ with
$$H_{\tau,k}( \sum_{i=0}^{m-1} a_i t^i )=\tau(a_0)+ \sum_{i=1}^{m-1}\tau(a_i) \big( \prod_{l=0}^{i-1} \sigma^l(k) \big) t^i,$$
for some $\tau\in {\rm Aut}_{S_0}(S)$ and some $k \in S^\times$  such that
 $\tau(d) = N_{S/S_0}(k) d$. All maps $H_{\tau,k}$ where $\tau \in {\rm Aut}_{S_0}(S)$ and where $k \in S^\times$  such that
 $\tau(d) = N_{S/S_0}(k) d$, are automorphisms of $A$.
\\ (ii) For all $k\in S^\times$ such that $N_{S/S_0}(k) = 1$, $id_S$ can be extended to an automorphism
$$H_{id,k}( \sum_{i=0}^{m-1}a_i t^i)= a_0+ \sum_{i=1}^{m-1} a_i \big( \prod_{l=0}^{i-1} \sigma^l(k) \big) t^i$$
in ${\rm Aut}_{S_0}(A)$.
\\ (iii) The subgroup of $S_0$-automorphisms of  $A$ extending $id_{S}$ is isomorphic to
$$\{ k \in S^\times\,|\, N_{S/S_0}(k) = 1 \}.$$
\end{theorem}

\begin{proof}
(i)  Let $S/S_0$ be a cyclic Galois ring extension with $\mathrm{Gal}(S/S_0) = \langle \sigma  \rangle$ of order $m$, $\sigma$ an automorphism of $S$. Then for all $\tau\in \mathrm{Gal}(S/S_0) $ and $k \in S^{\times}$  such that
\begin{equation}
\tau(d) = \Big( \prod_{l=0}^{m-1}\sigma^l(k) \Big) d,
\end{equation}
 the map $H_{\tau , k}:(S/S_0,\sigma,d)\longrightarrow (S/S_0,\sigma,d)$,
$$H_{\tau , k}(\sum_{i=0}^{m-1} x_i t^i )=\sum_{i=0}^{m-1} \tau(x_i) (kt)^i$$
 is an automorphism of the nonassociative cyclic Azumaya algebra $(S/S_0,\sigma,d)$ that extends $\tau$ (choose $D=C$ in Theorem \ref{thm:automorphism_of_Sf_field_case}). Since all $\tau\in {\rm Aut}_{S_0}(S)$ commute with $\sigma$, we obtain all automorphisms this way by Theorem \ref{thm:aut2}.
\\ (ii) and (iii) follow from (i) and Corollary \ref{cor:7}.
\end{proof}

 Note that if $S_0$ contains an $m$th root of unity $\omega$, and $\tau$ has finite order $s$, then the cyclic subgroup $\langle H_{\tau,\omega}\rangle$ of ${\rm Aut}_{S_0}((S/S_0,\sigma,d))$,  again has order at most $ms$.

If  $S_0$ has no non-trivial $m$th root of unity, we obtain:

\begin{theorem}\label{thm:inner}
Suppose $S_0$ has no non-trivial $m$th root of unity.
Let $A = (S/S_0,\sigma,d)$ be a nonassociative cyclic algebra of degree $m$ where $d \in S^{\times}$ is not contained
in any proper subring of $S$. 
Then every $S$-automorphism of $A$ leaves $S_0$ fixed and
$$\mathrm{Aut}_{S_0}(A) \cong \mathrm{ker}(N_{S/S_0}).$$
\end{theorem}

\begin{proof}
Every automorphism of $A$ has the form $H_{id,k}$:
suppose that there exist $j \in \{ 1, \ldots, m-1 \}$ and $k \in S^{\times}$ such that
$H_{\sigma^j,k} \in \mathrm{Aut}_{S_0}(A)$. This implies $H_{\sigma^j,k}^2 = H_{\sigma^j,k} \circ H_{\sigma^j,k} \in
\mathrm{Aut}_{S_0}(A)$  and
\begin{equation} \label{eqn:Aut(S_f) F does not contain primitive root of unity 1}
\begin{split}
H_{\sigma^j,k}^2 & \Big( \sum_{i=0}^{m-1} x_i t^i \Big) = \sigma^{2j}(x_0)  + \sum_{i=1}^{m-1} \sigma^{2j}(x_i) \Big( \prod_{q=0}^{i-1} \sigma^{j+q}(k) \sigma^q(k) \Big) t^i.
\end{split}
\end{equation}
Now $H_{\sigma^j,k}^2$ must have the form $H_{\sigma^{2j},l}$ for some $l \in S^{\times}$. Comparing
\eqref{eqn:neccessaryII} and \eqref{eqn:Aut(S_f) F does not contain primitive root of unity 1} yields
$l = k \sigma^j(k)$. Similarly, $H_{\sigma^j,k}^3 = H_{\sigma^{3j},s} \in \mathrm{Aut}_{S_0}(A)$ where
$s = k \sigma^j(k) \sigma^{2j}(k)$.
Continuing in this manner  the automorphisms $H_{\sigma^j,k}, H_{\sigma^{2j},l}, H_{\sigma^{3j},s},
\ldots$ all satisfy  (\ref{eqn:neccessary}) implying that
\begin{align} \label{eqn:Aut(S_f) F does not contain primitive root of unity 2}
\begin{split}
\sigma^j(d) &= N_{S/S_0}(k) d, \\
\sigma^{2j}(d) &= N_{S/S_0}(k \sigma^j(k)) d = N_{S/S_0}(k)^2 d, \\
\vdots & \qquad \qquad \vdots \\
d = \sigma^{n j}(d) &= N_{S/S_0}(k)^{n} d,
\end{split}
\end{align}
where $n = m/\mathrm{gcd}(j,m)$ is the order of $\sigma^j$.
Note that $\sigma^{ij}(d) \neq d$ for all $i \in \{ 1, \ldots, n-1 \}$ since $d$ is not contained in any proper
subring of $S$. Therefore $N_{S/S_0}(k)^{n} = 1$ and $N_{S/S_0}(k)^i \neq 1$ for all $i \in \{1, \ldots, n - 1 \}$
by \eqref{eqn:Aut(S_f) F does not contain primitive root of unity 2}, i.e. $N_{S/S_0}(k)$ is a primitive $n$th
 root of unity, thus also an $m$th root of unity, a contradiction. This proves the assertion.
\end{proof}


 \subsection{Associative generalized cyclic Azumaya algebras}


 In the associative setting, the previous results show that all automorphisms of a generalized cyclic Azumaya algebra $(D,\sigma,d)$ are induced by automorphisms of $D[t;\sigma]$:

\begin{corollary}  \label{cor:aut2}\label{cor:7new}
Let  $A=(D,\sigma,d)$ be a generalized cyclic  Azumaya algebra, i.e. $d\in S_0^\times$.
\\ (i) Every $\tau\in {\rm Aut}_{S_0}(D)$ that commutes with $\sigma$ can be extended to an automorphism
$$H_{\tau,k}( \sum_{i=0}^{m-1} a_i t^i )=\sum_{i=0}^{m-1} \tau(a_i) (kt)^i =\tau(a_0)+ \sum_{i=1}^{m-1}\tau(a_i) \big( \prod_{l=0}^{i-1} \sigma^l(k) \big) t^i$$
in $ {\rm Aut}_{S_0}(A)$
for some $k \in C^\times$ such that $ N_{C/S_0}(k)=1.$
All maps $H_{\tau,k}$ where $\tau \in {\rm Aut}_{S_0}(D)$ commutes with $\sigma$ and where $k \in C^\times$  such that
$ N_{C/S_0}(k)=1$ are automorphisms of $A$.
\\ (ii)  The subgroup of $S_0$-automorphisms extending some fixed $\tau\in {\rm Aut}_{S_0}(D)$ is isomorphic to
$$\{ k \in C^\times\,|\, N_{C/S_0}(k) = 1 \}.$$
 (iii) Suppose that $S_0$ contains a primitive $m$th root of unity $\omega$. If  $\tau$ has finite order $s$, then the cyclic subgroup $\langle H_{\tau,\omega}\rangle$ of ${\rm Aut}_{S_0}(A)$  generated by $H_{\tau,\omega}$ has order $ms$.
(If $\tau$ has infinite order then this subgroup has infinite order.)
Furthermore, $\langle H_{id,\omega}\rangle$
is a cyclic subgroup of ${\rm Aut}_{S_0}((D,\sigma,d))$ of order $m$.
\\ (iv) $\{\tau \in {\rm Aut}_{S_0}(D)\,|\, \tau\circ\sigma=\sigma\circ \tau \}$
 is isomorphic to a subgroup of ${\rm Aut}_{S_0}(A)$.
 \end{corollary}

\begin{proof}
It remains to show (iv), which holds since we have the canonical extension $H_{\tau,1}$ for each $\tau\in {\rm Aut}_{S_0}(A)$ that commutes with $\sigma$.
\end{proof}

\begin{corollary} \label{cor:aut3}
Let $A=(S/S_0,\sigma,d)$, $d\in S_0^\times$, be a cyclic Azumaya algebra. Then every $\tau\in {\rm Aut}_{S_0}(S)$ can be extended to an automorphism
$$H_{\tau,k}( \sum_{i=0}^{m-1} a_i t^i )=\tau(a_0)+ \sum_{i=1}^{m-1}\tau(a_i) \big( \prod_{l=0}^{i-1} \sigma^l(k) \big) t^i$$
in ${\rm Aut}_{S_0}(A)$
for some $k \in S^\times$ such that $ N_{S/S_0}(k)=1.$ All maps $H_{\tau,k}$ where $\tau \in {\rm Aut}_{S_0}(S)$ and where $k \in S^\times$  such that $ N_{S/S_0}(k) =1$ are algebra automorphisms of $A$.
\\ In particular, $ {\rm Aut}_{S_0}(S)=\langle\sigma \rangle\subset {\rm Aut}_{S_0}(A)$.
\end{corollary}

These are all possible automorphisms. So there is a bijection between the set of automorphisms of $A=(S/S_0,\sigma,d)$ and the set
$$\{(\tau,k)\,|\, \tau\in {\rm Aut}_{S_0}(A),\, k\in S^\times \text{ with }N_{S/S_0}(k)=1 \}$$
and for each $\tau \in {\rm Aut}_{S_0}(A)$, there are either infinitely many (if the set of norm one elements in $S$ is infinite) or $|\{k\in S^\times \,|\, N_{S/S_0}(k)=1\}|$ different possible extensions.

\section{Inner automorphisms}\label{sec:inner}

We now consider the inner automorphisms of nonassociative generalized cyclic Azumaya algebras.

\begin{theorem} \label{prop:inner}
 (a) Let $A=(D,\sigma,d)$ be a nonassociative generalized cyclic Azumaya algebra. Let $k\in C$ such that there is $c\in C^\times$ with $k=c^{-1}\sigma(c)$.
 \\ (i) $H_{\tau,k}=G_c\circ H_{\tau, 1}$ (note that $H_{\tau, 1}$ is not necessarily an automorphism here, just a map).
 \\ (ii) The automorphism $H_{id,k}$ of $A$ is the inner automorphism
$$G_c(\sum_{i=0}^{m-1}a_it^i)=(c^{-1}\sum_{i=0}^{m-1}a_it^i) c.$$
\\ (b) Let $A=(S/S_0,\sigma,d)$ be a nonassociative  cyclic Azumaya algebra. Let $k\in S$ such that there is $c\in S^\times$ with $k=c^{-1}\sigma(c)$.
 \\ (i) $H_{\tau,k}=G_c\circ H_{\tau, 1}$ (note that $H_{\tau, 1}$ again is not necessarily an automorphism here, just a map).
 \\ (ii) The automorphism $H_{id,k}$ of $A$ is the inner automorphism
$$G_c(\sum_{i=0}^{m-1}a_it^i)=(c^{-1}\sum_{i=0}^{m-1}a_it^i) c.$$
\end{theorem}

\begin{proof}
(a)
(i) For $k\in C$ such that $k=c^{-1}\sigma(c)$ for some $c\in C^\times$, we have
$$k\sigma(k)\cdots \sigma^{i-1}(k)=c\sigma^i(c),\quad i=1\dots,m-1,$$
 hence
$$G_c(\sum_{i=0}^{m-1}\tau(a_i)t^i)=(c^{-1}\sum_{i=0}^{m-1} \tau(a_i) t^i) c=\sum_{i=0}^{m-1}\tau(a_i) c^{-1}\sigma^i(c) t^i
=H_{\tau,k}(\sum_{i=0}^{m-1}a_it^i)$$
with the last equality holding because of $\Pi_{l=0}^{i-1}\sigma^l(c^{-1}\sigma(c))=c^{-1}\sigma^i(c).$
\\ (ii) follows from (i) and (b) from (a).
\end{proof}

 A similar result to Theorem \ref{prop:inner} (a) (i) was proved for automorphisms of (associative) $G$-Azumaya algebras over a connected commutative ring $R$ with a primitive $s$th root of unity in \cite{MB} when $G$ is a finite abelian group of finite order $n$ and exponent $s$, provided that ${\rm Pic}_s(R)$ is trivial and $n$ a unit in $R$.

Let us look at the associative setting:

\begin{theorem} \label{prop:gencyclicmain}
(a) Let $A=(D,\sigma,d)$ be a generalized cyclic Azumaya algebra.
\\ (i) Let $k\in C^\times$ such that there is $c\in C^\times$ with $k=c^{-1}\sigma(c)$. Then $H_{\tau,k}=G_c\circ H_{\tau, 1}$.
\\ (ii)  $H_{\sigma,1}=G_{t^{-1}}$ is an inner automorphism. Moreover,
$$\{G_{ct^{-j}}\,|\, c\in C^\times, 0\leq i\leq m-1\}$$
is a subgroup of ${\rm Aut}_{S_0}(A)$ of inner automorphisms.
\\ (b) Let $A=(S/S_0,\sigma,d)$ be a cyclic Azumaya algebra.
\\ (i) Let $k\in S^\times$ such that there is $c\in S^\times$ with $k=c^{-1}\sigma(c)$. Then $H_{\tau,k}=G_c\circ H_{\tau, 1}$.
\\ (ii) $\{G_{ct^{-j}}\,|\, c\in S^\times, 0\leq i\leq m-1\}$
is a subgroup of ${\rm Aut}_{S_0}(A)$ of inner automorphisms.
\end{theorem}

\begin{proof}
(a) (i) follows from Theorem \ref{prop:inner}.
\\ (ii) We know that $t^m=d$ in $A$. This means that $d^{-1}t^{m-1}$ is the inverse of $t$, since $d\in S_0$.
We have
$$G_{t^{-1}}(\sum_{i=0}^{m-1}a_it^i)=\sum_{i=0}^{m-1}t a_it^i(t^{m-1}d^{-1})=d^{-1}\sum_{i=0}^{m-1} \sigma(a_i)t^i d=H_{\sigma,1}(\sum_{i=0}^{m-1}a_it^i).$$
Thus also $H_{\sigma^j,1}=G_{t^{-j}}$ is an inner automorphism for all integers $j$, $0\leq j\leq m-1$, and so is
$G_c\circ H_{\sigma^j,1}=H_{\sigma^j,c^{-1}\sigma(c)}=G_{ct^{-j}}$ for all $c\in C^\times$. The rest of the assertion is trivial.
\\ All of (b) follows from (a).
\end{proof}

\begin{corollary} \label{thm:generalizedcycliccsa}
(i) Let  $(D,\sigma,d)$ be a generalized cyclic Azumaya algebra. If there exists an analogue of Hilbert's Theorem 90 for the ring extension $C/S_0$ (i.e., for every $k\in C$ with $N_{C/S_0}(k)=1$
there is $c\in C^\times$ such that $k=c^{-1}\sigma(c)$), then
$${\rm Aut}_{S_0}((D,\sigma,d))=\{H_{\tau,k}\,|\, \tau \in {\rm Aut}_{S_0}(D), \sigma\circ\tau=\tau\circ\sigma,\, k \in C^\times
\text{ such that } N_{C/S_0}(k)=1\}$$
$$=\{G_c\circ H_{\tau,1}\,|\, \tau \in {\rm Aut}_{S_0}(D), \sigma\circ\tau=\tau\circ\sigma,\, c \in C^\times \}.$$
 (ii) If there exists an analogue of Hilbert's Theorem 90 for the cyclic Galois ring extension $S/S_0$, then the cyclic Azumaya algebra $(S/S_0,\sigma,d)$ has the automorphism group
$${\rm Aut}_{S_0}((S/S_0,\sigma,d))=\{G_{ct^{-j}}\,|\, c\in S^\times, 0\leq i\leq m-1\}.$$
\end{corollary}

\begin{proof}
(i) The first equality is clear by Theorem \ref{thm:aut2}.
 By assumption, we can write $H_{\tau,k}=G_c(H_{\tau,k})$ for $k=\sigma(c)c^{-1}$, that is
 $$H_{\tau,k}(\sum_{i=0}^{m-1}a_it^i)=(c^{-1}\sum_{i=0}^{m-1}\tau(a_i)t^i) c$$
 which implies the second equality.
 \\ (ii) We know that ${\rm Aut}_{S_0}((S/S_0,\sigma,d))=\{H_{\sigma^j,k}\,|\, 0\leq j \leq m-1,\, k \in S^\times
\text{ such that } N_{S/S_0}(k)=1\}$. Now we also have $H_{\sigma^j,l}=G_c\circ H_{\sigma^j,1}$ for $l=\sigma(c)c^{-1}$, and $H_{\sigma^j,1}= G_{t^{-1}}\circ G_{t^{-1}}\circ \cdots G_{t^{-1}}=G_{t^{-j}}$, thus $H_{\sigma^j,l}=G_c\circ G_{t^{-j}}=G_{ct^{-j}}$ is an inner automorphism. This implies the assertion.
\end{proof}

\subsection{The automorphisms of central simple algebras}

In this section let $D$ be a division algebra which is finite-dimensional over its center $F={\rm C}(D)$  and $\sigma\in {\rm Aut}(D)$ such that $\sigma|_{F}$ has finite order $m$  and fixed field $F_0={\rm Fix}(\sigma)\cap F$. Thus $F/F_0$ is automatically a cyclic Galois field extension of degree $m$ with $\mathrm{Gal}(F/F_0) = \langle \sigma |_{F} \rangle$. For all $d\in F_0^\times$, $A=(D,\sigma,d)$ is a generalized cyclic central simple algebra over $F_0$, cf. \cite{J96}.

As an immediate consequence of our results, the automorphisms of a generalized cyclic algebra over a field are induced by ring automorphisms of the ring $D[t;\sigma]$ used in their construction. More precisely, they can be described as the composition of an inner automorphism $G_c$, $c\in C^\times$, with the canonical extension $H_{\tau,1}$ of some $\tau \in {\rm Aut}_{F_0}(D)$ which commutes with $\sigma$:

\begin{corollary}  \label{cor:generalizedcycliccsa}
(i) Let $A=(D,\sigma,d)$ be a central simple algebra over $F_0$, then
$${\rm Aut}_{F_0}(A)=\{H_{\tau,k}\,|\, \tau \in {\rm Aut}_{F_0}(D), \sigma\circ\tau=\tau\circ\sigma,\, k \in F^\times
\text{ such that } N_{F/F_0}(k)=1\}$$
$$=\{G_c\circ H_{\tau,1}\,|\, \tau \in {\rm Aut}_{F_0}(D), \sigma\circ\tau=\tau\circ\sigma,\, c \in F^\times \}.$$
(ii) Let $A=(K/F,\sigma,d)$ be a cyclic algebra over $F$ of degree $m$. Then
$${\rm Aut}_{F}(A)=\{G_{ct^{-j}}\,|\,  c \in K^\times,  0\leq j \leq m-1 \}.$$
 \end{corollary}


\end{document}